\documentclass{daj}

\usepackage{siunitx}
\usepackage{amssymb}
\usepackage{amsmath}
\usepackage{amsthm}
\usepackage{hyperref}

\newcommand{\E}{\mathbb{E}} 
\renewcommand{\P}{\mathbb{P}} 

\newcommand{\X}{\mathbf{X}}

\newcommand{\n}{\mathbf{n}} 
\newcommand{\g}{\mathbf{g}} 
\newcommand{\q}{\mathbf{q}} 
\renewcommand{\t}{\mathbf{t}} 
\newcommand{\h}{\mathbf{h}}

\theoremstyle{plain}

\newtheorem{theorem}{Theorem}[section]
\newtheorem{proposition}[theorem]{Proposition}

\newtheorem{lemma}[theorem]{Lemma}
\newtheorem{corollary}[theorem]{Corollary}

\theoremstyle{definition}

\newtheorem{remark}[theorem]{Remark}

\newtheorem{example}[theorem]{Example}

\newcommand\Z{\mathbb{Z}}
\newcommand\N{\mathbb{N}}
\newcommand\C{\mathbb{C}}

\newcommand\eps{\varepsilon}

\dajAUTHORdetails{%
  title = {The Erd\H{o}s Discrepancy Problem}, %% please capitalize all significant words
  author = {Terence Tao},
  plaintextauthor = {Terence Tao},
  plaintexttitle = {The Erdos Discrepancy Problem}, %%  title without math or LaTeX
  keywords = {discrepancy, multiplicative functions},
}   %%% END \dajAUTHORdetails

\dajEDITORdetails{%
   year={2016},
   number={1},
   received={17 September 2015},   % received date: example: 7 January 2017
   published={28 February 2016},  % published date
   doi={10.19086/da.609},       % XXX = number of paper, e.g. da006 for paper#6
}   %%% END \dajEDITORdetails

\begin{document}

\begin{frontmatter}[classification=text]
%% EDITOR: this will force the keywords to appear right after the Abstract.
%%   If the abstract is too long and would force the keywords off the
%%   front page, please comment out % [classification=text] above
%%   This way the keywords will be floated on the bottom of the first page
%%   even though the Abstract spills over to the next page.

%%% AUTHOR: Title goes here.  This line is optional.  You must use it
%%   if title has footnote attached or requires nontrivial typesetting,
%%   e.g., inclusion of linebreaks to force nice layout.
\title{The Erd\H{o}s discrepancy problem} %% please capitalize all significant words

%%% AUTHOR:
%%% List all authors. If you wish, place grant acknowledgements in \thanks.
%%% In brackets include a short tag for each author.
\author[tt]{Terence Tao\thanks{The author is supported by NSF grant DMS-0649473 and by a Simons Investigator Award.}}

%%% AUTHOR: Abstract goes here
\begin{abstract}
We show that for any sequence $f(1), f(2), \dots$ taking values in $\{-1,+1\}$, the discrepancy
$$ \sup_{n,d \in \N} \left|\sum_{j=1}^n f(jd)\right| $$
of $f$ is infinite. This answers a question of Erd\H{o}s.  In fact the argument also applies to sequences $f$ taking values in the unit sphere of a real or complex Hilbert space.  

The argument uses three ingredients.  The first is a Fourier-analytic reduction, obtained as part of the {\tt Polymath5} project on this problem, which reduces the problem to the case when $f$ is replaced by a (stochastic) completely multiplicative function $\g$.  The second is a logarithmically averaged version of the Elliott conjecture, established recently by the author, which effectively reduces to the case when $\g$ usually pretends to be a modulated Dirichlet character.  The final ingredient is (an extension of) a further argument obtained by the {\tt Polymath5} project which shows unbounded discrepancy in this case.
\end{abstract}
\end{frontmatter}

%%% AUTHOR: body of paper starts here
\section{Introduction}

Given a sequence $f\colon \N \to H$ taking values in a real or complex Hilbert space $H$, define the \emph{discrepancy} of $f$ to be the quantity
$$ \sup_{n,d \in \N} \left\| \sum_{j=1}^n f(jd) \right\|_H.$$
In other words, the discrepancy is the largest magnitude of a sum of $f$ along homogeneous arithmetic progressions $\{ d, 2d, \dots, nd\}$ in the natural numbers $\N = \{1,2,3,\dots\}$.

The main objective of this paper is to establish the following result:

\begin{theorem}[Erd\H{o}s discrepancy problem, vector-valued case]\label{hilbert}  Let $H$ be a real or complex Hilbert space, and let $f\colon \N \to H$ be a function such that $\|f(n)\|_H = 1$ for all $n$.  Then the discrepancy of $f$ is infinite.
\end{theorem}

Specialising to the case when $H$ is the reals, we thus have

\begin{corollary}[Erd\H{o}s discrepancy problem, original formulation]\label{edp}  Every sequence $f(1), f(2), \dots$ taking values in $\{-1,+1\}$ has infinite discrepancy.
\end{corollary}

This answers a question of Erd\H{o}s \cite{erdos} (see also \v{C}udakov \cite{cudakov} for related questions), which was recently the subject of the {\tt Polymath5} project \cite{polymath}; see the recent report \cite{gowers} on the latter project for further discussion.  

It is instructive to consider some near-counterexamples to these results -- that is to say, functions that are of unit magnitude, or nearly so, which have surprisingly small discrepancy -- to isolate the key difficulty of the problem.

\begin{example}[Dirichlet character]\label{dir}  Let $\chi \colon \N \to \C$ be a non-principal Dirichlet character of period $q$.  Then $\chi$ is completely multiplicative (thus $\chi(nm) = \chi(n) \chi(m)$ for any $n,m \in \N$) and has mean zero on any interval of length $q$.  Thus for any homogeneous arithmetic progression $\{d,2d,\dots,nd\}$ one has
$$ \left|\sum_{j=1}^n \chi(jd)\right| = \left|\chi(d) \sum_{j=1}^n \chi(j)\right| \leq q $$
and so the discrepancy of $\chi$ is at most $q$ (indeed one can refine this bound further using character sum bounds such as the Burgess bound \cite{burgess}).  Of course, this does not contradict Theorem \ref{hilbert} or Corollary \ref{edp}, even when the character $\chi$ is real, because $\chi(n)$ vanishes when $n$ shares a common factor with $q$.  Nevertheless, this demonstrates the need to exploit the hypothesis that $f$ has magnitude $1$ for \emph{all} $n$ (as opposed to merely for \emph{most} $n$).
\end{example}

\begin{example}[Borwein-Choi-Coons example]\label{bcc-ex}\cite{bcc}  Let $\chi_3$ be the non-principal Dirichlet character of period $3$ (thus $\chi_3(n)$ equals $+1$ when $n=1 \ (3)$, $-1$ when $n = 2 \ (3)$, and $0$ when $n = 0 \ (3)$), and define the completely multiplicative function $\tilde \chi_3\colon \N \to \{-1,+1\}$ by setting $\tilde \chi_3(p) := \chi_3(p)$ when $p \neq 3$ and $\tilde \chi_3(3) = +1$.  This is about the simplest modification one can make to Example \ref{dir} to eliminate the zeroes. Now consider the sum
$$ \sum_{j=1}^n \tilde \chi_3(j)$$
with $n := 1 + 3 + 3^2 + \dots + 3^k$ for some large $k$.  Writing $j = 3^i m$ with $m$ coprime to $3$ and $i$ at most $k$, we can write this sum as
$$ \sum_{i=0}^k \sum_{1 \leq m \leq n/3^j: (m,3)=1} \tilde \chi_3(3^i m).$$
Now observe that $\tilde \chi_3(3^i m) = \tilde \chi_3(3)^i \tilde \chi_3(m) = \chi_3(m)$.  The function $\chi_3$ has mean zero on every interval of length three, and $\lfloor n/3^j\rfloor$ is equal to $1$ mod $3$, hence
$$ \sum_{1 \leq m \leq n/3^j: (m,3)=1} \tilde \chi_3(3^i m) = 1$$
for every $i=0,\dots,k$.  Summing in $i$, we conclude that
$$ \sum_{j=1}^n \tilde \chi_3(j) = k+1 \gg \log n.$$
More generally, for natural numbers $n$, $\sum_{j=1}^n \tilde \chi_3(j)$ is equal to the number of $1$s in the base $3$ expansion of $n$.
Thus $\tilde \chi_3$ has infinite discrepancy, but the divergence is only logarithmic in the $n$ parameter; indeed from the above calculations and the complete multiplicativity of $\tilde \chi_3$ we see that $\sup_{n \leq N; d \in \N} | \sum_{j=1}^n \tilde \chi_3(jd)|$ is comparable to $\log N$ for $N>1$.  This can be compared with random sequences $f\colon \N \to \{-1,+1\}$, whose discrepancy would be expected to diverge like $N^{1/2+o(1)}$.  See the paper of Borwein, Choi, and Coons \cite{bcc} for further analysis of functions such as $\tilde \chi_3$, which seems to have been first discussed in \cite{schur}; see also \cite{but} for some further discussion of the sign patterns in $\tilde \chi_3$.  One can also reduce the discrepancy of this example slightly (by a factor of about two) by changing the value of the completely multiplicative function $\tilde \chi_3$ at $3$ from $+1$ to $-1$.

If one lets $\chi$ be a primitive Dirichlet character whose period is a prime $p = 1\ (4)$, and defines $\tilde \chi$ similarly to $\tilde \chi_3$ above, then the partial sums $\sum_{j=1}^n \tilde \chi(j)$ remain unbounded, but the Ces\`aro sum $\sum_{j=1}^n (1-\frac{j}{n}) \tilde \chi(j)$ remains bounded!  This observation\footnote{Bill Duke, personal communication} can be derived from a routine application of Perron's formula, together with the functional equation for $L(s,\chi)$, which in the $p=1\ (4)$ case establishes a zero of $L(s,\chi)$ at $s=0$.  Thus it is possible for the ``Ces\`aro smoothed discrepancy'' of a $\{-1,+1\}$-valued sequence to be finite.
\end{example}

\begin{example}[Vector-valued Borwein-Choi-Coons example]\label{vbc}  Let $H$ be a real Hilbert space with orthonormal basis $e_0,e_1,e_2,\dots$, let $\chi_3$ be the character from Example \ref{bcc-ex}, and let $f\colon \N \to H$ be the function defined by setting $f(3^a m) := \chi_3(m) e_a$ whenever $a=0,1,2\dots$ and $m$ is coprime to $3$.  Thus $f$ takes values in the unit sphere of $H$.   Repeating the calculation in Example \ref{bcc-ex}, we see that if $n = 1 + 3 + 3^2 + \dots + 3^k$, then
$$ \sum_{j=1}^n f(j) = e_0 + \dots + e_k$$
and hence
$$ \left\| \sum_{j=1}^n f(j) \right\|_H = \sqrt{k+1} \gg \sqrt{\log n}.$$
Conversely, if $n$ is a natural number and $d = 3^l d'$ for some $l=0,1,\dots$ and $d'$ coprime to $3$, we see from Pythagoras's theorem that
\begin{align*}
 \left\| \sum_{j=1}^n f(jd) \right\|_H &= \left\| \sum_{i \geq 0: 3^i \leq n} e_{i+l} \sum_{m \leq n/3^i} \chi_3(md') \right\|_H \\
&\leq \left(\sum_{i \geq 0: 3^i \leq n} 1\right)^{1/2}\\
&\ll \sqrt{\log n}.
\end{align*}
Thus the discrepancy of this function is infinite and diverges like $\sqrt{\log N}$.
\end{example}

\begin{example}[Random Borwein-Choi-Coons example]\label{rbc}  Let $\g\colon \N \to \{-1,+1\}$ be the stochastic (i.e. random) multiplicative function defined by setting $\g(n) = \chi_3(n)$ is coprime to $3$, and $\g(3^j) := \boldsymbol{\eps}_j$ for $j=1,2,3,\dots$, where $\chi_3$ is as in Example \ref{bcc-ex} and $\boldsymbol{\eps}_1, \boldsymbol{\eps}_2, \dots \in \{-1,+1\}$ are independently identically distributed signs, attaining $-1$ and $+1$ with equal probability.  Arguing similarly to \ref{vbc}, we have
\begin{align*}
\left(\E \left|\sum_{j=1}^n \g(jd)\right|^2\right)^{1/2} &= \left(\E \left| \sum_{i \geq 0: 3^i \leq n} \boldsymbol{\eps}_{i+l} \sum_{m \leq n/3^i} \chi_3(md') \right|^2\right)^{1/2} \\
&\leq \left(\sum_{i \geq 0: 3^i \leq n} 1\right)^{1/2}\\
&\ll \sqrt{\log n},
\end{align*}
where to reach the second line we use the additivity of variance for independent random variables.
Thus $\g$ in some sense has discrepancy growth like $\sqrt{\log N}$ ``on the average''.  (Note that one can interpret this example as a special case of Example \ref{vbc}, by setting $H$ to be the Hilbert space of real-valued square-integrable random variables.)  However, by carefully choosing the base $3$ expansion of $n$ depending on the signs $\boldsymbol{\eps}_1,\dots,\boldsymbol{\eps}_k$ (similarly to Example \ref{bcc-ex}) one can show that
$$ \sup_{n <3^{k+1}} \left|\sum_{j=1}^n \g(j)\right| \geq \frac{k+1}{2}$$
and so the actual discrepancy grows like $\log N$.  So this random example actually has essentially the same discrepancy growth as Example \ref{bcc-ex}.  We do not know if scalar sequences of significantly slower discrepancy growth than this can be constructed.
\end{example}

\begin{example}[Numerical examples] In \cite{kl} a sequence $f(n)$ supported on $n \leq N := \num{1160}$, with values in $\{-1,+1\}$ in that range, was constructed with discrepancy $2$ (and a SAT solver was used to show that $1160$ was the largest possible value of $N$ with this property).  A similar sequence with $N = \num{130000}$ of discrepancy $3$ was also constructed in that paper, as well as a sequence with $N = \num{127645}$ of discrepancy $3$ that was the restriction to $\{1,\dots,N\}$ of a completely multiplicative sequence taking values in $\{-1,+1\}$ (with the latter value of $\num{127645}$ being the best possible value of $N$; see also \cite{bgs} for a separate computation confirming this threshold).  This slow growth in discrepancy may be compared with the $\sqrt{\log N}$ type divergence in Example \ref{vbc}.
\end{example}

The above examples suggest that completely multiplicative functions are an important test case for Theorem \ref{hilbert} and Corollary \ref{edp}; the importance of this case was already isolated in \cite{erdos}.  More recently, the {\tt Polymath5} project \cite{polymath} obtained a number of equivalent formulations of the Erd\H{o}s discrepancy problem and its variants, including the logical equivalence of Theorem \ref{hilbert} with the following assertion involving such functions.  We define a \emph{stochastic} element of a measurable space $X$ to be a random variable $\g$ taking values in $X$, or equivalently a measurable map $g\colon \Omega \to X$ from an ambient probability space $(\Omega,\mu)$ (known as the \emph{sample space}) to $X$.

\begin{theorem}[Equivalent form of vector-valued Erd\H{o}s discrepancy problem]\label{vec}  Let $\g\colon \N \to S^1$ be a \emph{stochastic} completely multiplicative function taking values in the unit circle $S^1 := \{z \in \C: |z|=1\}$ (where we give the space $(S^1)^\N$ of functions from $\N$ to $S^1$ the product $\sigma$-algebra).  Then
$$ \sup_n \E \left|\sum_{j=1}^n \g(j)\right|^2 = +\infty.$$
\end{theorem}

By converting all the probabilistic language to measure-theoretic language, Theorem \ref{vec} has the following equivalent form:

\begin{theorem}[Measure-theoretic formulation]\label{mes}  Let $(\Omega,\mu)$ be a probability space, and let $g: \Omega \to (S^1)^\N$ be a measurable function to the space $(S^1)^\N$ of functions from $\N$ to $S^1$, such that $g(\omega) \in (S^1)^\N$ is completely multiplicative for $\mu$-almost every $\omega \in \Omega$ (that is to say, $g(\omega)(nm) = g(\omega)(n) g(\omega)(m)$ for all $n,m \in \N$ and $\mu$-almost all $\omega \in \Omega$).  Then one has
$$ \sup_n \int_\Omega \left|\sum_{j=1}^n g(\omega)(j)\right|^2\ d\mu(\omega) = +\infty.$$
\end{theorem}

The equivalence between Theorem \ref{hilbert} and Theorem \ref{vec} (or Theorem \ref{mes}) was obtained in \cite{polymath} using a Fourier-analytic argument; for the convenience of the reader, we reproduce this argument in Section \ref{fourier-sec}.  The close similarity between Example \ref{vbc} and Example \ref{rbc} can be interpreted as a special case of this equivalence.

It thus remains to establish Theorem \ref{vec}.  To do this, we use a recent result of the author \cite{tao-elliott} regarding correlations of multiplicative functions:

\begin{theorem}[Logarithmically averaged nonasymptotic Elliott conjecture]\label{elliott}\cite[Theorem 1.3]{tao-elliott}   Let $a_1,a_2$ be natural numbers, and let $b_1,b_2$ be integers such that $a_1 b_2 - a_2 b_1 \neq 0$.   Let $\eps > 0$, and suppose that $A$ is sufficiently large depending on $\eps,a_1,a_2,b_1,b_2$.  Let $x \geq w \geq A$, and let $g_1,g_2\colon \N \to \C$ be multiplicative functions with $|g_1(n)|, |g_2(n)| \leq 1$ for all $n$, with $g_1$ ``non-pretentious'' in the sense that
\begin{equation}\label{tax0}
 \sum_{p \leq x} \frac{1 - \operatorname{Re} g_1(p) \overline{\chi(p)} p^{-it}}{p} \geq A 
\end{equation}
for all Dirichlet characters $\chi$ of period at most $A$, and all real numbers $t$ with $|t| \leq Ax$.  Then
\begin{equation}\label{mang} \left|\sum_{x/w < n \leq x} \frac{g_1(a_1 n + b_1) g_2(a_2 n + b_2)}{n}\right| \leq \eps \log \omega.
\end{equation}
\end{theorem}

This theorem is a variant of the Elliott conjecture \cite{elliott} (as corrected in \cite{mrt}), which in turn is a generalisation of a well known conjecture of Chowla \cite{chowla}.  See \cite{tao-elliott} for further discussion of this result, the proof of which relies on a number of tools, including the recent results in \cite{mr}, \cite{mrt} on mean values of multiplicative functions in short intervals.  It can be viewed as a sort of ``inverse theorem'' for pair correlations of multiplicative functions, asserting that such correlations can only be large when both of the multiplicative functions ``pretend'' to be like modulated Dirichlet characters $n \mapsto \chi(n) n^{it}$.

Using this result and a standard van der Corput argument, one can show that the only potential counterexamples to Theorem \ref{vec} come from (stochastic) completely multiplicative functions that usually ``pretend'' to be like modulated Dirichlet characters (cf. Examples \ref{dir}, \ref{bcc-ex}, \ref{rbc}).  More precisely, we have

\begin{proposition}[van der Corput argument]\label{star}  Suppose that $\g\colon \N \to S^1$ is a stochastic completely multiplicative function, such that
\begin{equation}\label{nax}
\E \left|\sum_{j=1}^n \g(j)\right|^2 \leq C^2
\end{equation}
for some finite $C>0$ and all natural numbers $n$ (thus, $\g$ would be counterexample to Theorem \ref{vec}).  Let $\eps > 0$, and suppose that $X$ is sufficiently large depending on $\eps,C$.  Then with probability $1-O(\eps)$, one can find a (stochastic) Dirichlet character $\boldsymbol{\chi}$ of period $\mathbf{q} = O_{C,\eps}(1)$ and a (stochastic) real number $\t = O_{C,\eps}(X)$ such that
\begin{equation}\label{gh0}
 \sum_{p \leq X} \frac{1 -\operatorname{Re} \g(p) \overline{\boldsymbol{\chi}(p)} p^{-i\t}}{p} \ll_{C,\eps} 1.
\end{equation}
\end{proposition}

(See Section \ref{notation-sec} below for our asymptotic notation conventions.)  We give the (easy) derivation of Proposition \ref{star} from Theorem \ref{elliott} in Section \ref{elliott-sec}.  One can of course reformulate Proposition \ref{star} in measure-theoretic language if desired, much as Theorem \ref{vec} may be reformulated as Theorem \ref{mes}; we leave this to the interested reader.  Of course, Theorem \ref{vec} implies that the hypotheses of Proposition \ref{star} cannot hold, and so Proposition \ref{star} is in fact vacuously true; nevertheless it is necessary to establish this proposition independently of Theorem \ref{vec} to avoid circularity.

It remains to demonstrate Theorem \ref{vec} for random completely multiplicative functions $\g$ that obey \eqref{gh0} with high probability for large $X$ and small $\eps$.  Such functions $\g$ can be viewed as (somewhat complicated) generalisations of the Borwein-Choi-Coons example (Example \ref{bcc-ex}), and it turns out that a more complicated version of the analysis in Example \ref{bcc-ex} (or Example \ref{vbc}) suffices to establish a lower bound for $\E |\sum_{j=1}^n \g(j)|^2$ (of logarithmic type, similar to that in Example \ref{vbc}) which is enough to conclude Theorem \ref{vec} and hence Theorem \ref{hilbert} and Corollary \ref{edp}.  We give this argument in Section \ref{bcc-sec}.

In principle, the arguments in \cite{tao-elliott} provide an effective value for $A$ as a function of $\eps,a_1,a_2,b_1,b_2$ in Theorem \ref{elliott}, which would in turn give an explicit lower bound for the divergence of the discrepancy in Theorem \ref{hilbert} or Corollary \ref{edp}.  However, this bound is likely to be far too weak to match the $\sqrt{\log N}$ type divergence observed in Example \ref{vbc}.  Nevertheless, it seems reasonable to conjecture that the $\sqrt{\log N}$ order of divergence is best possible for Theorem \ref{hilbert} (although it is unclear to the author whether such a slowly diverging example can also be attained for Corollary \ref{edp}).

The arguments in this paper can also be used to partially classify the multiplicative (but not completely multiplicative) functions taking values in $\{-1,+1\}$ that have bounded partial sums; see Section \ref{sumsec}. 

\begin{remark}  In \cite{polymath, gowers}, Theorem \ref{hilbert} was also shown to be equivalent to the existence of sequences $(c_{m,d})_{m,d \in \N}$, $(b_n)_{n \in \N}$ of non-negative reals such that $\sum_{m,d} c_{m,d}=1$, $\sum_n b_n = \infty$, and such that the real quadratic form
$$ \sum_{m,d} c_{m,d}(x_d+x_{2d}+\dots+x_{md})^2 - \sum_n b_n x_n^2$$
is positive semi-definite.  The arguments of this paper thus abstractly show that such sequences exist, but do not appear to give any explicit construction for such a sequence.
\end{remark}

\begin{remark}  In \cite[Conjecture 3.12]{gowers}, the following stronger version of Theorem \ref{hilbert} was proposed: if $C \geq 0$ and $N$ is sufficiently large depending on $C$, then for any matrix $(a_{ij})_{1 \leq i,j \leq N}$ of reals with diagonal entries equal to $1$, there exist homogeneous arithmetic progressions $P = \{d,2d,\dots,nd\}$ and $Q = \{d',2d',\dots,n'd'\}$ in $\{1,\dots,N\}$ such that
$$ |\sum_{i \in P} \sum_{j \in Q} a_{ij}| \geq C.$$
Setting $a_{ij} := \langle f(i), f(j) \rangle_H$, we see that this would indeed imply Theorem \ref{hilbert} and thus Corollary \ref{edp}.  We do not know how to resolve this conjecture, although it appears that a two-dimensional variant of the Fourier-analytic arguments in Section \ref{fourier-sec} below can handle the special case when $a_{ij} = \pm 1$ for all $i,j$ (which would still imply Corollary \ref{edp} as a special case).  We leave this modification of the argument to the interested reader.
\end{remark}

\begin{remark}  As a further near-counterexample to Corollary \ref{edp}, we present here an example of a sequence $f: \N \to \{-1,+1\}$ for which
$$ \sup_{n \in \N} \left| \sum_{j=1}^n f(jd)\right| < \infty$$
for each $d$ (though of course the left-hand side must be unbounded in $d$, thanks to Corollary \ref{edp}).  We set $f(1) := 1$, and then recursively for $D = 1, 2, \dots$ we define
$$f(jD! + k) := (-1)^j f(k)$$
for $k=1,\dots,D!$ and $j=1,\dots,D$.  Thus for instance the first few elements of the sequence\footnote{OEIS A262725} are
$$1, -1, -1, 1, 1, -1, -1, 1, 1, -1, -1, 1, 1, -1, -1, 1, 1, -1, \dots.$$
If $d$ is a natural number and $D$ is any odd number larger than $d$, then we see that the block $f(d), f(2d), \dots, f( (D+1)! )$ is $D+1$ alternating copies of $f(d), f(2d), \dots, f(D!)$ and thus sums to zero; in fact if we divide $f(d), f(2d), \dots$ into consecutive blocks of length $(D+1)!/d$ then all such blocks sum to zero, and so $\sup_n |\sum_{j=1}^n f(jd)|$ is finite for all $d$.
\end{remark}

\begin{remark}  There is a curious superficial similarity between the arguments in this paper and the Hardy-Littlewood circle method.  In the latter, Fourier analytic arguments are used to reduce matters to estimates on ``major arcs'' and ``minor arcs''; in this paper, Fourier analytic arguments are used to reduce matters to estimates for ``pretentious multiplicative functions'' and ``non-pretentious multiplicative functions''.  We do not know if there is any deeper significance to this similarity.
\end{remark}

\subsection{Notation}\label{notation-sec}

We adopt the usual asymptotic notation of $X \ll Y$, $Y \gg X$, or $X = O(Y)$ to denote the assertion that $|X| \leq CY$ for some constant $C$.  If we need $C$ to depend on an additional parameter we will denote this by subscripts, e.g. $X = O_\eps(Y)$ denotes the bound $|X| \leq C_\eps Y$ for some $C_\eps$ depending on $\eps$.   
For any real number $\alpha$, we write $e(\alpha) := e^{2\pi i \alpha}$.

All sums and products will be over the natural numbers $\N =\{1,2,\dots\}$ unless otherwise specified, with the exception of sums and products over $p$ which is always understood to be prime.

We use $d|n$ to denote the assertion that $d$ divides $n$, and $n\ (d)$ to denote the residue class of $n$ modulo $d$.  We use $(a,b)$ to denote the greatest common divisor of $a$ and $b$.

We will frequently use probabilistic notation such as the expectation $\E \X$ of a random variable $\X$ or a probability $\P(E)$ of an event $E$.  We will use boldface symbols such as $\g$ to refer to random (i.e. stochastic) variables, to distinguish them from deterministic variables, which will be in non-boldface.

\section{Fourier analytic reduction}\label{fourier-sec}

In this section we establish the logical equivalence between Theorem \ref{hilbert} and Theorem \ref{vec} (or Theorem \ref{mes}).  The arguments here are taken from a website\footnote{{\tt michaelnielsen.org/polymath1/index.php?title=Fourier\_reduction}} of the {\tt Polymath5} project \cite{polymath}.

The deduction of Theorem \ref{mes} from Theorem \ref{hilbert} is straightforward: if $(\Omega,\mu)$ and $g$ are as in Theorem \ref{mes}, one takes $H$ to be the complex Hilbert space $L^2(\Omega,\mu)$, and for each natural number $n$, we let $f(n) \in H$ be the function
$$ f(n): \omega \mapsto g(\omega)(n).$$
Since $g(\omega)(n) \in S^1$, $f(n)$ is clearly a unit vector in $H$.  For any homogeneous arithmetic progression $\{d,2d,\dots,nd\}$, one has
\begin{align*}
\left\| \sum_{j =1}^n f(jd)\right\|_H^2 &= \int_\Omega \left|\sum_{j=1}^n g(\omega)(jd)\right|^2\ d\mu(\omega)\\
&= \int_\Omega \left|\sum_{j=1}^n g(\omega)(j)\right|^2\ d\mu(\omega)
\end{align*}
and on taking suprema in $n$ and $d$ we conclude that Theorem \ref{mes} follows from Theorem \ref{hilbert}.  (Note that this argument also explains the similarity between Example \ref{rbc} and Example \ref{vbc}.)

Since Theorem \ref{mes} is equivalent to Theorem \ref{vec}, it remains to show that Theorem \ref{vec} implies Theorem \ref{hilbert}.
We take contrapositives, thus we assume that Theorem \ref{hilbert} fails, and seek to conclude that Theorem \ref{vec} also fails.  By hypothesis, we can find a function $f\colon \N \to H$ taking values in the unit sphere of a Hilbert space $H$ and a finite quantity $C$ such that
\begin{equation}\label{jn}
 \left\|\sum_{j=1}^n f(jd)\right\|_H \leq C
\end{equation}
for all homogeneous arithmetic progressions $d,2d,\dots,nd$.  By complexifying $H$ if necessary, we may take $H$ to be a complex Hilbert space.  To obtain the required conclusion, it will suffice to construct a random completely multiplicative function $\g$ taking values in $S^1$, such that
$$ \E \left|\sum_{j=1}^n \g(j)\right|^2 \ll_C 1$$
for all $n$.

We claim that it suffices to construct, for each $X \geq 1$, a stochastic completely multiplicative function $\g_X$ taking values in $S^1$ such that
\begin{equation}\label{gca}
 \E \left|\sum_{j=1}^n \g_X(j)\right|^2 \ll_C 1
\end{equation}
for all $n \leq X$, where the implied constant is uniform in $n$ and $X$, but we allow the underlying probability space defining the stochastic function $\g_X$ to depend on $X$.  This reduction is obtained by a standard compactness argument\footnote{If one wished to obtain a more quantitative version of Theorem \ref{hilbert}, one would avoid this compactness argument and work instead with truncated versions of Theorem \ref{vec} (or Theorem \ref{mes}) in which one restricts the $n$ parameter to be less than some large cutoff.  This would then require similar truncations to be made in the arguments in later sections, which in particular requires some treatment of error terms created when truncating Euler products, but such errors can be made negligible by making the truncation parameter extremely large with respect to all other parameters.  We leave the details of this reformulation of the argument to the interested reader.}, but we give the details for sake of completeness.  Suppose that for each $X$, we have such a $\g_X$ obeying \eqref{gca} as above.  Let $\mathcal{M}$ be the space of completely multiplicative functions $g: \N \to S^1$ from $\N$ to $S^1$; one can view this space as isomorphic to an infinite product of $S^1$'s, since completely multiplicative functions are determined by their values at the primes.  In particular, $\mathcal{M}$ is a compact metrisable space; it can be viewed as a compact subspace of the space $(S^1)^\N$ of arbitrary functions (not necessarily multiplicative) from $\N$ to $S^1$.

Just as Theorem \ref{vec} is equivalent to Theorem \ref{mes}, we can view each $\g_X$ as a measurable map $f_X: \Omega_X \to \mathcal{M}$, such that
$$ \int_{\Omega_X} \left|\sum_{j=1}^n f_X(\omega)(j)\right|^2\ d\mu_X(\omega) \ll_C 1$$
for all $n \leq X$.  We can then define a Radon probability measure $\nu_X$ on $\mathcal{M}$ to be the probability distribution (or law) of the random variable $\g_X$, or equivalently the pushforward of the measure $\mu_X$ via $f_X$.  That is to say,
\begin{align*}
\int_{\mathcal{M}} F(g)\ d\nu_X(g) &= {\bf E} F( \g_X ) \\
&= \int_{\Omega_X} F(f_X(\omega))\ d\mu_X(\omega)
\end{align*}
for any continuous function $F: \mathcal{M} \to \C$.
The functions $g \mapsto \left|\sum_{j=1}^n g(j)\right|^2$ are continuous on $\mathcal{M}$, and hence
$$ \int_{\mathcal{M}} \left|\sum_{j=1}^n g(j)\right|^2\ d\nu_X(g) \ll_C 1$$
for all $n \leq X$.  By vague compactness of probability measures on compact metrisable spaces such as $\mathcal{M}$ (Prokhorov's theorem), we can thus extract a subsequence $\nu_{X_j}$ of the $\nu_X$ with $X_j \to \infty$ such that the $\nu_{X_j}$ converge to a Radon probability measure $\nu$ on $\mathcal{M}$, that is to say
$$ \int_{\mathcal{M}} F(g)\ d\nu_{X_j}(g) \to \int_{\mathcal{M}} F(g)\ d\nu_X(g) $$
as $j \to \infty$ for all continuous functions $F: \mathcal{M} \to \C$.  Applying this in particular to the continuous functions $g \mapsto \left|\sum_{j=1}^n g(j)\right|^2$, we conclude that
$$ \int_{\mathcal{M}} \left|\sum_{j=1}^n g(j)\right|^2\ d\nu(g) \ll_C 1$$
for \emph{all} $n$.  We then define the random completely multiplicative function $\g: \N \to S^1$ (or equivalently, a measurable map from a probability space to $\mathcal{M}$) by choosing $(\mathcal{M}, \nu)$ as the underlying probability space, and using the identity function $g \mapsto g$ as the measurable map.  We then have
$$ \E \left|\sum_{j=1}^n \g(j)\right|^2 \ll_C 1$$
for all $\n$, and the claim follows.

It remains to construct the random multiplicative functions $\g_X$ for each $X$.  Let $X \geq 1$, and let $p_1,\dots,p_r$ be the primes up to $X$.  Let $M \geq X$ be a natural number that we assume to be sufficiently large depending on $C,X$. Define a function $F \colon (\Z/M\Z)^r \to H$ by the formula
$$ F( a_1\ (M),\dots,a_r\ (M)) := f( p_1^{a_1} \dots p_r^{a_r} )$$
for $a_1,\dots,a_r \in \{0,\dots,M-1\}$, thus $F$ takes values in the unit sphere of $H$.  We also define the function $\pi \colon [1,X] \to (\Z/M\Z)^r$ by setting $\pi( p_1^{a_1} \dots p_r^{a_r} ) := (a_1,\dots,a_r)$ whenever $p_1^{a_1} \dots p_r^{a_r}$ is in the discrete interval
$$[1,X] := \{ n \in \N: 1 \leq n \leq X\};$$ 
note that $\pi$ is well defined for $M \geq X$.  Applying \eqref{jn} for $n \leq X$ and $d$ of the form $p_1^{a_1} \dots p_r^{a_r}$ with $1 \leq a_i \leq M - X$, we conclude that
$$ \left\|\sum_{j=1}^n F( x + \pi(j) )\right\|_H \ll_C 1$$
for all $n \leq X$ and all but $O_X( M^{r-1} )$ of the $M^r$ elements $x = (x_1,\dots,x_r)$ of $(\Z/M\Z)^r$.  For the exceptional elements, we have the trivial bound
$$ \left\|\sum_{j=1}^n F( x + \pi(j) )\right\|_H \leq n \leq X$$
from the triangle inequality.
Square-summing in $x$, we conclude (if $M$ is sufficiently large depending on $C,X$) that
\begin{equation}\label{fpi}
\frac{1}{M^r} \sum_{x \in (\Z/M\Z)^r} \left\|\sum_{j=1}^n F( x + \pi(j) )\right\|_H^2 \ll_C 1.
\end{equation}
By Fourier expansion, we can write
$$ F(x) = \sum_{\xi \in (\Z/M\Z)^r} \hat F(\xi) e\left( \frac{x \cdot \xi}{M} \right)$$
where $(x_1,\dots,x_r) \cdot (\xi_1,\dots,\xi_r) := x_1 \xi_1 + \dots + x_r \xi_r$, and the Fourier transform $\hat F \colon (\Z/M\Z)^r \to H$ is defined by the formula
$$ \hat F(\xi) := \frac{1}{M^r} \sum_{x \in (\Z/M\Z)^r} F(x) e\left( -\frac{x \cdot \xi}{M} \right).$$
A routine Fourier-analytic calculation (using the Plancherel identity) then allows us to write the left-hand side of \eqref{fpi} as
$$
\sum_{\xi \in (\Z/M\Z)^r} \|\hat F(\xi)\|_H^2 \left|\sum_{j=1}^n e\left( \frac{\pi(j) \cdot \xi}{M} \right)\right|^2.$$
On the other hand, from a further application of the Plancherel identity we have
$$\sum_{\xi \in (\Z/M\Z)^r} \|\hat F(\xi)\|_H^2 = 1$$
and so we can interpret $\|\hat F(\xi)\|_H^2$ as the probability distribution of a random frequency $\boldsymbol{\xi} = (\boldsymbol{\xi}_1,\dots,\boldsymbol{\xi}_r) \in (\Z/M\Z)^r$ (using $(\Z/M\Z)^r$ as the underlying sample space).  The estimate \eqref{fpi} now takes the form
$$ \E \left|\sum_{j=1}^n e\left( \frac{\pi(j) \cdot \boldsymbol{\xi}}{M} \right)\right|^2 \ll_C 1$$
for all $n \leq X$.
If we then define the stochastic completely multiplicative function $\g_X$ by setting $\g_X(p_j) := e( \boldsymbol{\xi}_j / M )$ for $j=1,\dots,r$, and $\g_X(p) := 1$ for all other primes, we obtain 
$$ \E \left|\sum_{j=1}^n \g_X(j)\right|^2 \ll_C 1$$
for all $n \leq X$, as desired.

\begin{remark} It is instructive to see how the above argument breaks down when one tries to use the Dirichlet character example in Example \ref{dir}.  While $\chi$ often has magnitude $1$ in the ordinary (Archimedean) sense, the function $(a_1,\dots,a_r) \mapsto \chi(p_1^{a_1} \dots p_r^{a_r})$ is almost always zero, since the argument $p_1^{a_1} \dots p_r^{a_r}$ of $\chi$ is likely to be a multiple of $q$.  As such, the quantity $\| \hat F(\xi)\|_H^2$ sums to something much less than $1$, and one does not generate a stochastic completely multiplicative function $\g$ with bounded discrepancy.
\end{remark}

\begin{remark} The above arguments also show that Theorem \ref{hilbert} automatically implies an apparently stronger version\footnote{This version was suggested by Harrison Brown in the {\tt Polymath5} project.} of itself, in which one assumes $\|f(n)\|_H \geq 1$ for all $n$, rather than $\|f(n)\|_H = 1$.  Indeed, if $f$ has bounded discrepancy then it must be bounded (since $f(n)$ is the difference of $\sum_{j=1}^n f(j)$ and $\sum_{j=1}^{n-1} f(j)$), and the above arguments then carry through; the sum $\sum_{\xi \in (\Z/M\Z)^r} \|\hat F(\xi)\|_H^2$ is now greater than or equal to $1$, but one can still define a suitable probability distribution from the $\|\hat F(\xi)\|_H^2$ by normalising.
\end{remark}

\begin{remark} If Theorem \ref{vec} failed, then we could find a constant $C>0$ and a stochastic completely multiplicative function $\g\colon \N \to S^1$ such that
$$ \E \left|\sum_{j=1}^n \g(j)\right|^2 \leq C^2
$$
for all $n$.  In particular, by the triangle inequality we have
$$ \E \frac{1}{N} \sum_{n=1}^N \left|\sum_{j=1}^n \g(j)\right|^2 \leq C^2
$$
and hence for each $N$, there exists a \emph{deterministic} completely multiplicative function $g_N\colon \N \to S^1$ such that
$$ \frac{1}{N} \sum_{n=1}^N \left|\sum_{j=1}^n g_N(j)\right|^2 \leq C^2.
$$
Thus, to prove Theorem \ref{vec} (and hence Theorem \ref{hilbert} and Corollary \ref{edp}), it would suffice to obtain a lower bound of the form
\begin{equation}\label{nan}
 \frac{1}{N} \sum_{n=1}^N \left|\sum_{j=1}^n g(j)\right|^2 > \omega(N)
\end{equation}
for \emph{all} deterministic completely multiplicative functions $g\colon \N \to S^1$, all $N \geq 1$, and some function $\omega(N)$ of $N$ that goes to infinity as $N \to \infty$.  This was in fact the preferred form of the Fourier-analytic reduction obtained by the {\tt Polymath5} project \cite{polymath}, \cite{gowers}.  It is conceivable that some refinement of the analysis in this paper in fact yields a bound of the form \eqref{nan}, though this seems to require removing the logarithmic averaging from Theorem \ref{elliott}, as well as avoiding the use of Lemma \ref{tb} below.
\end{remark}

\section{Applying the Elliott-type conjecture}\label{elliott-sec}

In this section we prove Proposition \ref{star}.  Let $\g, C, \eps$ be as in that proposition.
Let $H \geq 1$ be a moderately large natural number depending on $\eps$ to be chosen later, and suppose that $X$ is sufficiently large depending on $H,\eps$.   From \eqref{nax} and the triangle inequality we have
$$ \E \sum_{\sqrt{X} \leq n \leq X} \frac{1}{n} \left|\sum_{j=1}^n \g(j)\right|^2 \ll_C \log X;$$
a similar argument (for $X$ large enough) gives
$$ \E \sum_{\sqrt{X} \leq n \leq X} \frac{1}{n} \left|\sum_{j=1}^{n+H} \g(j)\right|^2 \ll_C \log X,$$
and hence by the triangle inequality
$$ \E \sum_{\sqrt{X} \leq n \leq X} \frac{1}{n} \left|\sum_{j=n+1}^{n+H} \g(j)\right|^2 \ll_C \log X.$$
Thus from Markov's inequality we see with probability $1-O(\eps)$ that
$$ \sum_{\sqrt{X} \leq n \leq X} \frac{1}{n} \left|\sum_{j=n+1}^{n+H} \g(j)\right|^2 \ll_{C,\eps} \log X,$$
which we rewrite as
\begin{equation}\label{nxx}
\sum_{\sqrt{X} \leq n \leq X} \frac{1}{n} \left|\sum_{h=1}^{H} \g(n+h)\right|^2 \ll_{C,\eps} \log X.
\end{equation}
We can expand out the left-hand side of \eqref{nxx} as
$$ \sum_{h_1, h_2 \in [1,H]} \sum_{\sqrt{X} \leq n \leq X} \frac{\g(n+h_1) \overline{\g(n+h_2)}}{n}.$$
The diagonal term $h_1,h_2$ contributes a term of size $\gg H \log X$ to this expression.  Thus, choosing $H$ to be a sufficiently large quantity depending on $C,\eps$, we can apply the triangle inequality and pigeonhole principle to find \emph{distinct} (and stochastic) $\mathbf{h}_1,\mathbf{h}_2 \in [1,H]$ such that
$$ \left|\sum_{\sqrt{X} \leq n\leq X} \frac{\g(n+\mathbf{h}_1) \overline{\g(n+\mathbf{h}_2)}}{n}\right| \gg_{C,\eps,H} \log X.$$
Applying Theorem \ref{elliott} in the contrapositive, we obtain the claim.  (It is easy to check that the quantities $\chi, t$ produced by Theorem \ref{elliott} can be selected to be measurable, for instance one can use continuity to restrict $t$ to be rational and then take a minimal choice of $(\chi,t)$ with respect to some explicit well-ordering of the countable set of possible pairs $(\chi,t)$.)

\begin{remark}\label{gil}  The same argument shows that the hypothesis $|\g(n)| = 1$ may be relaxed to $|\g(n)| \leq 1$, and $\g$ need only be multiplicative rather than completely multiplicative, provided that one has a lower bound of the form $\sum_{\sqrt{X} \leq n \leq X} \frac{|\g(n)|^2}{n} \gg \log X$.  Thus the Dirichlet character example in Example \ref{dir} is in some sense the ``only'' example of a bounded multiplicative function with bounded discrepancy that is large for many values of $n$, in that any other such example must ``pretend'' to be like a (modulated) Dirichlet character.  (We thank Gil Kalai for suggesting this remark.)
\end{remark}

\section{A generalised Borwein-Choi-Coons analysis}\label{bcc-sec}

We can now complete the proof of Theorem \ref{vec} (and thus Theorem \ref{hilbert} and Corollary \ref{edp}).  Our arguments here will be based on those from a website\footnote{{\tt michaelnielsen.org/polymath1/index.php?title=} {\tt Bounded\_discrepancy\_multiplicative\_functions\_do\_not\_correlate\_with\_characters}} of the {\tt Polymath5} project \cite{polymath}, which treated the case in which the functions $\g$ and $\boldsymbol{\chi}$ appearing in Proposition \ref{star} were real-valued (and the quantity $\t$ was set to zero).

Suppose for contradiction that Theorem \ref{vec} failed\footnote{Readers who are more comfortable with measure-theoretic notation than probabilistic notation may prefer to write the argument below starting from the failure of Theorem \ref{mes} rather than Theorem \ref{vec}, replacing expectations with integrals, etc.}, then we can find a constant $C>0$ and a stochastic completely multiplicative function $\g\colon \N \to S^1$ such that
$$
\E \left|\sum_{j=1}^n \g(j)\right|^2 \leq C^2
$$
for all natural numbers $n$.  We now allow all implied constants to depend on $C$, thus
$$
\E \left|\sum_{j=1}^n \g(j)\right|^2 \ll 1$$
for all $n$.  The stochastic nature of $\g$ is a mild technical nuisance for our arguments, but the reader may wish to assume $\g$ as a deterministic completely multiplicative function for a first reading, as this case already captures the key aspects of the argument.

We will need the following large and small parameters, selected in the following order:
\begin{itemize}
\item A quantity $0 < \eps < 1/2$ that is sufficiently small depending on $C$.
\item A natural number $H \geq 1$ that is sufficiently large depending on $C,\eps$.
\item A quantity $0 < \delta < 1/2$ that is sufficiently small depending on $C,\eps,H$.
\item A natural number $k \geq 1$ that is sufficiently large depending on $C,\eps,H$.
\item A real number $X \geq 1$ that is sufficiently large depending on $C,\eps,H,\delta,k$.
\end{itemize}
We will implicitly assume these size relationships in the sequel to simplify the computations, for instance by absorbing a smaller error term into a larger if the latter dominates the former under the above assumptions.  The reader may wish to keep the hierarchy
$$ C \ll \frac{1}{\eps} \ll H \ll \frac{1}{\delta}, k \ll X $$
in mind in the arguments that follow.  One could reduce the number of parameters in the argument by setting $\delta := 1/k$, but this does not lead to significant simplifications in the arguments below.

By Proposition \ref{star}, we see with probability $1-O(\eps)$ that there exists a Dirichlet character $\boldsymbol{\chi}$ of period $\q=O_\eps(1)$ and a real number $\t = O_\eps(X)$ such that
\begin{equation}\label{gh}
  \sum_{p \leq X} \frac{1 - \operatorname{Re} \g(p) \overline{\boldsymbol{\chi}(p)} p^{-i\t}}{p} \ll_\eps 1.
	\end{equation}
By reducing $\boldsymbol{\chi}$ if necessary we may assume that $\boldsymbol{\chi}$ is primitive.

It will be convenient to cut down the size of $\t$.

\begin{lemma}\label{tb} With probability $1-O(\eps)$, one has 
\begin{equation}\label{t-bound}
\t = O_{\eps}(X^\delta).
\end{equation}
\end{lemma}

\begin{proof} By Proposition \ref{star} with $X$ replaced by $X^\delta$, we see that with probability $1-O(\eps)$, one can find a Dirichlet character $\boldsymbol{\chi}'$ of period $\q' = O_\eps(1)$ and a real number $\t' = O_\eps( X^\delta )$ such that
$$
  \sum_{p \leq X^\delta} \frac{1 - \operatorname{Re} \g(p) \overline{\boldsymbol{\chi}'(p)} p^{-i\t'}}{p} \ll_\eps 1.$$
We may restrict to the event that $|\t'-\t| \geq X^\delta$, since we are done otherwise.
Applying the pretentious triangle inequality (see \cite[Lemma 3.1]{gs}), we conclude that
\begin{equation}\label{psum}
  \sum_{p \leq X^\delta} \frac{1 - \operatorname{Re} \boldsymbol{\chi}(p) \overline{\boldsymbol{\chi}'(p)} p^{-i(\t'-\t)}}{p} \ll_\eps 1.
\end{equation}
The character $\boldsymbol{\chi} \overline{\boldsymbol{\chi}'}$ has period $O_\eps(1)$.  Applying the Vinogradov-Korobov zero-free region for $L(\cdot,\boldsymbol{\chi} \overline{\boldsymbol{\chi}'})$ (see \cite[\S 9.5]{mont}), we see that $L(\sigma+it, \boldsymbol{\chi} \overline{\boldsymbol{\chi}'}) \neq 0$ for $|t| \geq 10$ and
$$ \sigma \geq 1 - \frac{c_\eps}{(\log |t|)^{2/3} (\log\log |t|)^{1/3}}$$
for some $c_\eps>0$ depending only on $\eps$; furthermore, an inspection of the Vinogradov-Korobov arguments (based on estimation of the logarithmic derivative of $L(\cdot,\boldsymbol{\chi} \overline{\boldsymbol{\chi}'})$ in the zero-free region) in fact yields the crude bound\footnote{One can certainly improve the right-hand side here with a more careful argument; cf. \cite[(11.6)]{mv}.  But for the current application, a logarithmic bound will suffice.}
$$ |\log L(\sigma+it, \boldsymbol{\chi} \overline{\boldsymbol{\chi}'})| \ll \log^{O(1)} |t|$$
in this region (using a suitable branch of the logarithm), possibly after shrinking $c_\eps$ if necessary.    Using the contour-shifting arguments in \cite[Lemma 2]{mr-short} 
and the bounds $X^\delta \leq |\t'-\t| \ll_\eps X$, it is then not difficult to show that
$$
 \sum_{\exp((\log X)^{2/3}) \leq p \leq X^\delta} \frac{1 - \operatorname{Re}\boldsymbol{\chi}(p) \overline{\boldsymbol{\chi}'(p)} p^{-i(\t'-\t)}}{p} \gg \log\log X$$
if $X$ is sufficiently large depending on $\eps,\delta$, a contradiction (note that the summands in \eqref{psum} are nonnegative).  The claim follows.
\end{proof}
	
Let us now condition to the probability $1-O(\eps)$ event that $\boldsymbol{\chi}$, $\mathbf{t}$ exist obeying \eqref{gh} and the bound \eqref{t-bound}; we can of course do this as $\eps$ is assumed to be small.

The bound \eqref{gh} asserts that $\g$ ``pretends'' to be like the completely multiplicative function $n \mapsto \boldsymbol{\chi}(n) n^{i\t}$.  We can formalise this by making the factorisation
\begin{equation}\label{dos}
 \g(n) := \tilde{\boldsymbol{\chi}}(n) n^{i\t} \h(n) 
\end{equation}
where $\tilde{\boldsymbol{\chi}}$ is the completely multiplicative function of magnitude $1$ defined by setting $\tilde{\boldsymbol{\chi}}(p) := \boldsymbol{\chi}(p)$ for $p \nmid q$ and $\tilde{\boldsymbol{\chi}}(p) := \g(p) p^{-i\t}$ for $p|q$, and $\h$ is the completely multiplicative function of magnitude $1$ defined by setting $\h(p) := \g(p) \overline{\boldsymbol{\chi}(p)} p^{-i\t}$ for $p \nmid q$, and $h(p) = 1$ for $p|q$.  The function $\tilde{\boldsymbol{\chi}}$ should be compared with the function $\tilde \chi_3$ in Example \ref{bcc-ex} and the function $\g$ in Example \ref{rbc}.

With the above notation, the bound \eqref{gh} simplifies to
\begin{equation}\label{gap}
 \left|\sum_{p \leq X} \frac{1-\operatorname{Re} \h(p)}{p}\right| \ll_\eps 1.
\end{equation}

The model case to consider here is when $\t = 0$ and $\h = 1$, in which case $\g = \tilde \chi$.  In this case, one could skip directly ahead to \eqref{contra} below.  Of course, in general $\t$ will be non-zero (albeit not too large) and $\h$ will not be identically $1$ (but ``pretends'' to be $1$ in the sense of \eqref{gap}).   We will now perform some manipulations to remove the $n^{i\t}$ and $\h$ factors from $\g$ and isolate medium-length sums \eqref{contra} of the $\tilde{\boldsymbol{\chi}}$ factor, which are more tractable to compute with than the corresponding sums of $\g$; then we will perform more computations to arrive at an expression \eqref{tex} just involving $\boldsymbol{\chi}$ which we will be able to control fairly easily.  

We turn to the details.  The first step is to eliminate the role of $n^{i\t}$.   From \eqref{nax} and the triangle inequality we have
$$
\E \frac{1}{H} \sum_{H < H' \leq 2H} \left|\sum_{m=1}^{H'} \g(n+m)\right|^2 \ll 1 $$
for all $n$ (even after conditioning to the $1-O(\eps)$ event mentioned above).  The $\frac{1}{H} \sum_{H < H' \leq 2H}$ averaging will not be used until much later in the argument, and the reader may wish to ignore it for the time being.

By \eqref{dos}, the above estimate can be written as
$$
\E \frac{1}{H} \sum_{H < H' \leq 2H} \left|\sum_{m=1}^{H'} \tilde{\boldsymbol{\chi}}(n+m) (n+m)^{i\t} \h(n+m)\right|^2 \ll 1.$$
For $n \geq X^{2\delta}$, we can use \eqref{t-bound} and Taylor expansion to conclude that $(n+m)^{i\t} = n^{i\t} + O_{\eps,H,\delta}(X^{-\delta})$.  The contribution of the error term is negligible, thus
$$\E \frac{1}{H} \sum_{H < H' \leq 2H} \left|\sum_{m=1}^{H'}\tilde{\boldsymbol{\chi}}(n+m) n^{i\t} \h(n+m)\right|^2 \ll 1$$
for all $n \geq X^{2\delta}$.  We can factor out the $n^{i\t}$ factor to obtain
$$\E \frac{1}{H} \sum_{H < H' \leq 2H} \left|\sum_{m=1}^{H'} \tilde{\boldsymbol{\chi}}(n+m) h(n+m)\right|^2 \ll 1.$$
For $n < X^{2\delta}$ we can crudely bound the left-hand side by $H^2$.  If $\delta$ is sufficiently small, we can then sum weighted by $\frac{1}{n^{1+1/\log X}}$ and conclude that
$$\E \frac{1}{H} \sum_{H < H' \leq 2H}  \sum_n \frac{ \left|\sum_{m=1}^{H'} \tilde{\boldsymbol{\chi}}(n+m) \h(n+m)\right|^2}{n^{1+1/\log X}}  \ll \log X.$$
(The zeta function type weight of $\frac{1}{n^{1+1/\log X}}$ will be convenient later in the argument when one has to perform some multiplicative number theory, as the relevant sums can be computed quite directly and easily using Euler products.)
Thus, with probability $1-O(\eps)$, one has from Markov's inequality that
$$ \frac{1}{H} \sum_{H < H' \leq 2H} \sum_n \frac{\left|\sum_{m=1}^{H'} \tilde{\boldsymbol{\chi}}(n+m) \h(n+m)\right|^2}{n^{1+1/\log X}}  \ll_\eps \log X.$$
We condition to this event, which we may do as $\eps$ is assumed to be small.  From this point onwards, our arguments will be purely deterministic in nature (in particular, one can ignore the boldface fonts in the arguments below if one wishes).

We have successfully eliminated the role of $n^{i\t}$; we now work to eliminate $\h$.  To do this we will have to partially decouple the $\tilde{\boldsymbol{\chi}}$ and $\h$ factors in the above expression, which can be done\footnote{The argument here was loosely inspired by the Maier matrix method \cite{maier}.} by exploiting the almost periodicity properties of $\tilde{\boldsymbol{\chi}}$ as follows.
Call a residue class $a \ (\q^k)$ \emph{bad} if $a+m$ is divisible by $p^k$ for some $p|\q$ and $1 \leq m \leq 2H$, and \emph{good} otherwise.  We restrict $n$ to good residue classes, thus
$$ \frac{1}{H} \sum_{H < H' \leq 2H}\sum_{a \in [1,\q^k], \hbox{ good}} \sum_{n = a \ (\q^k)} \frac{\left|\sum_{m=1}^{H'}\tilde{\boldsymbol{\chi}}(n+m) \h(n+m)\right|^2}{n^{1+1/\log X}} $$
$$\ll_\eps \log X.$$
By Cauchy-Schwarz, we conclude that
\begin{align*}
&\frac{1}{H} \sum_{H < H' \leq 2H} \sum_{a \in [1,\q^k], \hbox{ good}} \left |\sum_{n = a \ (\q^k)} \frac{ \sum_{m=1}^{H'} \tilde{\boldsymbol{\chi}}(n+m) \h(n+m)}{n^{1+1/\log X}} \right|^2\\
&\quad\quad  \ll_\eps \frac{ \log^2 X}{\q^k}.
\end{align*}
Now we claim that for $n$ in a given good residue class $a \ (\q^k)$, the quantity $\tilde{\boldsymbol{\chi}}(n+m)$ does not depend on $n$.  Indeed, by hypothesis, $(n+m,\q^k) = (a+m,\q^k)$ is not divisible by $p^k$ for any $p|\q$ and is thus a factor of $\q^{k-1}$, and $\frac{n+m}{(n+m,\q^k)} = \frac{n+m}{(a+m,\q^k)}$ is coprime to $\q$.  We then factor
\begin{align*}
\tilde{\boldsymbol{\chi}}(n+m)&= \tilde{\boldsymbol{\chi}}((n+m,\q^k)) \tilde{\boldsymbol{\chi}}\left( \frac{n+m}{(n+m,\q^k)} \right)\\
&= \tilde{\boldsymbol{\chi}}((a+m,\q^k)) \boldsymbol{\chi}\left( \frac{n+m}{(a+m,\q^k)} \right)\\
&= \tilde{\boldsymbol{\chi}}((a+m,\q^k)) \boldsymbol{\chi}\left( \frac{a+m}{(a+m,\q^k)} \right)
\end{align*}
where in the last line we use the periodicity of $\boldsymbol{\chi}$.  Thus we have $\tilde{\boldsymbol{\chi}}(n+m) = \tilde{\boldsymbol{\chi}}(a+m)$, and so
\begin{align*}
&\frac{1}{H} \sum_{H < H' \leq 2H} \sum_{a \in [1,\q^k], \hbox{ good}}\left |\sum_{m=1}^{H'} \tilde{\boldsymbol{\chi}}(a+m) \sum_{n = a \ (\q^k)} \frac{\h(n+m)}{n^{1+1/\log X}}\right |^2\\
&\quad\quad \ll_\eps \frac{ \log^2 X}{\q^k}.
\end{align*}
Shifting $n$ by $m$, we see that
$$ \sum_{n = a \ (\q^k)} \frac{ \h(n+m) }{n^{1+1/\log X}}= \sum_{n = a+m \ (\q^k)} \frac{\h(n)}{n^{1+1/\log X}}  + O_H(1) $$
and thus (for $X$ large enough)
\begin{equation}\label{jock}
\begin{split}
&\frac{1}{H} \sum_{H < H' \leq 2H} \sum_{a \in [1,\q^k], \hbox{ good}} \left|\sum_{m=1}^{H'} \tilde{\boldsymbol{\chi}}(a+m) \sum_{n = a+m \ (\q^k)} \frac{\h(n)}{n^{1+1/\log X}}\right|^2 \\
&\quad \ll_\eps \frac{\log^2 X}{\q^k} .
\end{split}
\end{equation}

Now, we perform some multiplicative number theory to understand the innermost sum in \eqref{jock}, with the aim of showing that the summand here is approximately equidistributed modulo $\q^k$.  From taking Euler products, we have
$$
\sum_n \frac{\h(n)}{n^{1+1/\log X}} = \boldsymbol{\mathfrak S}$$
where $\boldsymbol{\mathfrak S}$ is the Euler product 
$$\boldsymbol{\mathfrak S} := \prod_p \left(1 - \frac{\h(p)}{p^{1+1/\log X}}\right)^{-1}.$$
From \eqref{gap} and Mertens' theorem one can easily verify that
\begin{equation}\label{1s}
\log X \ll_\eps |\boldsymbol{\mathfrak S}| \ll_\eps \log X.
\end{equation}
More generally, for any Dirichlet character $\chi_1$ we have
$$
\sum_n \frac{\chi_1(n) \h(n)}{n^{1+1/\log X}}  =  \prod_p \left(1 - \frac{\h(p) \chi_1(p)}{p^{1+1/\log X}}\right)^{-1}.$$
If $\chi_1$ is a non-principal character of period dividing $\q^k$, then the $L$-function $L(s,\chi) := \sum_n \frac{\chi_1(n)}{n^s}$ is analytic near $s=1$, and in particular we have
$$ L( 1+\frac{1}{\log X}, \chi_1) \ll_{\q,k} 1.$$
We conclude that
\begin{align*}
\sum_n \frac{\chi_1(n) \h(n)}{n^{1+1/\log X}}  &=  L(1+\frac{1}{\log X}, \chi_1) \prod_p \left(1 - \frac{\h(p) \chi_1(p)}{p^{1+1/\log X}}\right)^{-1} \left(1 - \frac{\chi_1(p)}{p^{1+1/\log X}}\right) \\
&\ll_{\q,k} \exp\left( \sum_p \frac{|1-\h(p)|}{p^{1+1/\log X}} \right) \\
&\ll_{\q,k} \exp\left( \sum_{p \leq X} \frac{|1-\h(p)|}{p} \right) \\
&\ll_{\q,k} \exp\left( \sum_{p \leq X} \frac{O( 1-\operatorname{Re} \h(p))^{1/2}}{p}   \right) \\
&\ll_{\q,k} \exp\left( O\left( (\log\log X) \sum_{p \leq X} \frac{1-\operatorname{Re} \h(p)}{p}\right)^{1/2} \right) \\
&\ll_{\q,k} \exp( O_\eps( (\log\log X)^{1/2} ) )
\end{align*}
where we have used the Cauchy-Schwarz inequality, Mertens' theorem, and \eqref{gap}. For a principal character $\chi_0$ of period $r$ dividing $\q^k$ we have
\begin{align*}
\sum_n \frac{\chi_0\h(n)}{n^{1+1/\log X}} &= \prod_{p \nmid r} \left(1 - \frac{\h(p)}{p^{1+1/\log X}}\right)^{-1}\\
&= \boldsymbol{\mathfrak S} \prod_{p | r} \left(1 - \frac{1}{p^{1+1/\log X}}\right) \\
&= \boldsymbol{\mathfrak S} \left(1 + O_\eps\left( \frac{1}{\log X} \right) \right) \prod_{p | r} \left(1 - \frac{1}{p}\right) \\
&= \frac{\phi(r)}{r} \boldsymbol{\mathfrak S} + O_{\eps}(1)
\end{align*}
thanks to \eqref{1s} and the fact that $\h(p)=1$ for all $p|r$, and that all prime factors of $r$ divide $\q$ and are thus of size $O_\eps(1)$.  By expansion into Dirichlet characters we conclude that
$$
\sum_{n = b\ (r)} \frac{\h(n)}{n^{1+1/\log X}} = \frac{\boldsymbol{\mathfrak S}}{r}  + O_{\q,k}(\exp(O_\eps((\log\log X)^{1/2})))$$
for all $r|q^k$ and primitive residue classes $b \ (r)$.  For non-primitive residue classes $b \ (r)$, we write $r = (b,r) r'$ and $b = (b,r) b'$.  The previous arguments then give
$$
\sum_{n = b' \ (r')} \frac{\h(n)}{n^{1+1/\log X}} = \frac{\boldsymbol{\mathfrak S}}{r'}  + O_{\q,k}(\exp(O_\eps((\log\log X)^{1/2})))$$
which since $\h((b,r))=1$ gives (again using \eqref{1s})
$$
\sum_{n = b \ (r)} \frac{\h(n)}{n^{1+1/\log X}} = \frac{\boldsymbol{\mathfrak S}}{r}  + O_{\q,k}(\exp(O_\eps((\log\log X)^{1/2})))$$
for all $b \ (r)$ (not necessarily primitive).  Inserting this back into \eqref{jock} we see that
$$
\frac{1}{H} \sum_{H < H' \leq 2H}\sum_{a \in [1,\q^k] \hbox{ good}} \left| \sum_{m=1}^{H'}\tilde{\boldsymbol{\chi}}(a+m) \left(\frac{\boldsymbol{{\mathfrak S}}}{\q^k}  + O_{\q,k}(\exp(O_\eps((\log\log X)^{1/2})))\right) \right|^2 \ll_\eps \frac{\log^2 X}{\q^k}.$$
The contribution of the $O_{\q,k}(\exp(O_\eps((\log\log X)^{1/2})))$ error term here can be shown by \eqref{1s} to be at most $c_\eps \log^2 X / \q^k$ in magnitude if $X$ is large enough, for any $c_\eps>0$ depending only on $\eps$.  Removing this error term and then applying \eqref{1s} again to cancel off the ${\mathfrak S}$ term, we conclude that
\begin{equation}\label{contra}
\frac{1}{\q^k} \sum_{a \in [1,\q^k] \hbox{ good}} \frac{1}{H} \sum_{H < H' \leq 2H} \left|\sum_{m=1}^{H'} \tilde{\boldsymbol{\chi}}(a+m)\right|^2 \ll_\eps 1.
\end{equation}

We have now eliminated both $\t$ and $\h$.  The remaining task is to establish some lower bound on the discrepancy of medium-length sums of $\tilde{\boldsymbol{\chi}}$ that will contradict \eqref{contra}.  As mentioned above, this will be a more complicated variant of the analysis in Examples \ref{bcc-ex}, \ref{vbc}, \ref{rbc} in which the perfect orthogonality in Example \ref{vbc} is replaced by an almost orthogonality argument.

We turn to the details.  We first dispose of the easy case\footnote{In this case, many of the previous manipulations become degenerate, and one could have disposed of this case by a simplified version of the above arguments.} when $\q=1$.  In that case $\tilde{\boldsymbol{\chi}}$ is identically one, and the left-hand side simplifies to $\frac{1}{H} \sum_{H < H' \leq 2H} (H')^2$, which is comparable to $H^2$ and leads to a contradiction since $H$ is large.    Thus we may restrict to the event that $\q > 1$, so that the primitive character $\boldsymbol{\chi}$ is non-principal.

Next, we expand \eqref{contra} to obtain
$$
\frac{1}{H} \sum_{H < H' \leq 2H} \sum_{m_1,m_2 \in [1,H']} \sum_{a \in [1,\q^k], \hbox{ good}} \tilde{\boldsymbol{\chi}}(a+m_1) \overline{\tilde{\boldsymbol{\chi}}(a+m_2)} \ll_\eps \q^k.$$
Write $d_1 := (a+m_1,\q^k)$ and $d_2 := (a+m_2,\q^k)$, thus $d_1,d_2 | \q^{k-1}$ and for $i=1,2$ we have
$$ \tilde{\boldsymbol{\chi}}(a+m_i) = \tilde{\boldsymbol{\chi}}(d_i) \boldsymbol{\chi}\left( \frac{a+m_i}{d_i} \right).$$
We thus have
\begin{equation}\label{stop}
\begin{split}
&\sum_{d_1,d_2|\q^{k-1}} \tilde{\boldsymbol{\chi}}(d_1) \overline{\tilde{\boldsymbol{\chi}}}(d_2)
\frac{1}{H} \sum_{H < H' \leq 2H} \sum_{m_1,m_2 \in [1,H']} \\
&\quad \sum_{a \in [1,\q^k], \hbox{ good}: (a+m_1,\q^k)=d_1, (a+m_2,\q^k)=d_2} \boldsymbol{\chi}\left(\frac{a+m_1}{d_1}\right) \overline{\boldsymbol{\chi}}\left(\frac{a+m_2}{d_2}\right) \ll_\eps \q^k.
\end{split}
\end{equation}
We reinstate the bad $a$.  The number of such $a$ is at most
$$ H \sum_{p|\q} p^{-k} \q^k \ll H 2^{-k} \q^k \sum_{n \geq 2} \frac{1}{(n/2)^k} \ll H 2^{-k} \q^k,$$
so their total contribution here is $O_H(2^{-k} \q^k)$ which is negligible, thus we may drop the requirement in \eqref{stop} that $a$ is good.

Note that as $\boldsymbol{\chi}$ is already restricted to numbers coprime to $\q$, and $d_1,d_2$ divide $\q^{k-1}$, we may replace the constraints $(a+m_i,\q^k)=d_i$ with $d_i|a+m_i$ for $i=1,2$.  Summarising these modifications, we have arrived at the estimate
\begin{equation}\label{stop-2}
\begin{split}
&\sum_{d_1,d_2|\q^{k-1}} \tilde{\boldsymbol{\chi}}(d_1) \overline{\tilde{\boldsymbol{\chi}}}(d_2)
\frac{1}{H} \sum_{H < H' \leq 2H} \sum_{m_1,m_2 \in [1,H']} \\
&\quad \sum_{a \in [1,\q^k]: d_1|a+m_1; d_2|a+m_2} \boldsymbol{\chi}\left(\frac{a+m_1}{d_1}\right) \overline{\boldsymbol{\chi}}\left(\frac{a+m_2}{d_2}\right) \ll_\eps \q^k.
\end{split}
\end{equation}

Consider the contribution to the left-hand side of \eqref{stop-2} of an off-diagonal term $d_1 \neq d_2$ for a fixed choice of $m_1,m_2$.  To handle these terms we use the Fourier transform to expand the  character $\boldsymbol{\chi}(n)$ (which, as mentioned before Lemma \ref{tb}, can be taken to be primitive) as a linear combination of $e( \xi n / \q )$ for $\xi \in (\Z/\q\Z)^\times$.  Thus, the function $n \mapsto 
1_{d_1|n} \boldsymbol{\chi}\left(\frac{n}{d_1}\right)$ can be written as a linear combination of $n \mapsto 1_{d_1|n} e(\xi n/d_1 \q)$ for $\xi \in \Z$ coprime to $\q$, which by Fourier expansion of the $1_{d_1|n}$ factor (and the fact that all the prime factors of $d_1$ also divide $\q$) can in turn be written as a linear combination of $n \mapsto e(\xi n/d_1\q)$ for $(\xi,d_1\q)=1$.  Translating, we see that the function
$$ a \mapsto 1_{d_1|a+m_1} \boldsymbol{\chi}\left(\frac{a+m_1}{d_1}\right) $$
can be written as a linear combination of $a \mapsto e(\xi a/d_1\q)$ for $(\xi,d_1\q)=1$.  Similarly
$$ a \mapsto 1_{d_2|a+m_2} \boldsymbol{\chi}\left(\frac{a+m_2}{d_2}\right) $$
can be written as a linear combination of $a \mapsto e(\xi a/d_2\q)$ for $(\xi,d_2\q)=1$.  If $d_1 \neq d_2$, then the frequencies involved here are distinct; since $\q^k$ is a multiple of both $d_1\q$ and $d_2\q$, we conclude the perfect cancellation\footnote{We thank Andrew Granville for observing this perfect cancellation, which allowed for some simplifications to this part of the argument.}
\begin{equation}\label{perf}
 \sum_{a \in [1,\q^k]: d_1|a+m_1; d_2|a+m_2} \boldsymbol{\chi}\left(\frac{a+m_1}{d_1}\right) \overline{\boldsymbol{\chi}}\left(\frac{a+m_2}{d_2}\right) = 0.
\end{equation}
Thus we only need to consider the diagonal contribution $d_1 = d_2$ to \eqref{stop-2}.  For these diagonal terms we do not perform a Fourier expansion of the character $\boldsymbol{\chi}$.  The $ \tilde{\boldsymbol{\chi}}(d_1) \overline{\tilde{\boldsymbol{\chi}}}(d_2)$ terms helpfully cancel, and we obtain the bound
\begin{equation}\label{tex}
\begin{split}
& \sum_{d|\q^{k-1}}
\frac{1}{H} \sum_{H < H' \leq 2H} \sum_{m_1,m_2 \in [1,H']} \sum_{a \in [1,\q^k]: d | a+m_1, a+m_2}\\
&\quad \boldsymbol{\chi}\left(\frac{a+m_1}{d}\right) \overline{\boldsymbol{\chi}}\left(\frac{a+m_2}{d}\right) \ll_\eps \q^k.
\end{split}
\end{equation}
We have now eliminated $\tilde{\boldsymbol{\chi}}$, leaving only the Dirichlet character $\boldsymbol{\chi}$ which is much easier to work with.
We gather terms and write the left-hand side as
$$
 \sum_{d|\q^{k-1}}
\frac{1}{H} \sum_{H < H' \leq 2H} \sum_{a \in [1,\q^k]} \left|\sum_{m \in [1,H']: d|a+m} \boldsymbol{\chi}\left(\frac{a+m}{d}\right)\right|^2.$$
The summand in $d$ is now non-negative.  We can thus discard all the $d$ that are not of the form $d = \q^i$ with $\q^i < \sqrt{H}$, to conclude that
$$ \sum_{i: \q^i < \sqrt{H}} \frac{1}{H} \sum_{H < H' \leq 2H} \sum_{a \in [1,\q^k]} \left|\sum_{m \in [1,H']: \q^i|a+m} \boldsymbol{\chi}\left(\frac{a+m}{\q^i}\right)\right|^2  \ll_\eps \q^k.$$
It is now that we finally take advantage of the averaging $\frac{1}{H} \sum_{H < H' \leq 2H}$ to simplify the $m$ summation.  Observe from the triangle inequality that for any $H' \in [H, 3H/2]$ and $a \in [1,\q^k]$ one has
\begin{align*}
& \left|\sum_{H' < m \leq H' + \q^i: \q^i|a+m} \boldsymbol{\chi}\left(\frac{a+m}{\q^i}\right)\right|^2\\
&\quad  \ll \left|\sum_{m \in [1,H']: \q^i|a+m} \boldsymbol{\chi}\left(\frac{a+m}{\q^i}\right)\right|^2 + \left|\sum_{m \in [1,H'+\q^i]: \q^i|a+m} \boldsymbol{\chi}\left(\frac{a+m}{\q^i}\right)\right|^2;
\end{align*}
summing over $i,H',a$ we conclude that
$$ \sum_{i: \q^i < \sqrt{H}} \frac{1}{H} \sum_{H' \in [H,3H/2]} \sum_{a \in [1,\q^k]} \left |\sum_{H' < m \leq H'+\q^i: \q^i|a+m} \boldsymbol{\chi}\left(\frac{a+m}{\q^i}\right)\right|^2  \ll_\eps \q^k.$$
In particular, by the pigeonhole principle there exists $\mathbf{H}' \in [H,3H/2]$ such that
$$ \sum_{i: \q^i < \sqrt{H}} \sum_{a \in [1,\q^k]}\left |\sum_{\mathbf{H}' < m \leq \mathbf{H}'+\q^i: \q^i|a+m} \boldsymbol{\chi}\left(\frac{a+m}{\q^i}\right)\right|^2  \ll_\eps \q^k.$$
Shifting $a$ by $\mathbf{H}'$ and discarding some terms, we conclude that
$$ \sum_{i: \q^i < \sqrt{H}} \sum_{a \in [1,\q^k/2]} \left|\sum_{0 < m \leq \q^i: \q^i|a+m} \boldsymbol{\chi}\left(\frac{a+m}{\q^i}\right)\right|^2  \ll_\eps \q^k.$$
Observe that for a fixed $a$ there is exactly one $m$ in the inner sum, and $\frac{a+m}{\q^i} = \left\lfloor \frac{a}{\q^i} \right\rfloor + 1$.  Thus we have
$$ \sum_{i: \q^i < \sqrt{H}} \sum_{a \in [1,\q^k/2]} \left|\boldsymbol{\chi}\left(\left\lfloor \frac{a}{\q^i}\right\rfloor + 1\right)\right|^2  \ll_\eps \q^k.$$
Making the change of variables $b := \left\lfloor \frac{a}{\q^i} \right\rfloor + 1$ and discarding some terms, we thus have
$$ \sum_{i: \q^i < \sqrt{H}} \q^i \sum_{b \in [1,\q^{k-i}/4]} |\boldsymbol{\chi}(b)|^2 \ll_\eps \q^k.$$
But $b \mapsto |\boldsymbol{\chi}(b)|^2$ is periodic of period $\q$ with mean $\gg_\eps 1$, thus
$$  \sum_{b \in [1,\q^{k-i}/4]} |\boldsymbol{\chi}(b)|^2 \gg_\eps \q^{k-i}$$
which when combined with the preceding bound yields
$$ \sum_{i: \q^i < \sqrt{H}} 1 \ll_\eps 1,$$
which leads to a contradiction for $H$ large enough (note the logarithmic growth in $H$ here, which is consistent with the growth rates in Example \ref{vbc}).  The claim follows.

\section{Sums of multiplicative functions}\label{sumsec}

One corollary of Corollary \ref{edp} is that if $f: \N \to \{-1,+1\}$ is a completely multiplicative function, then
$$ \sup_n |\sum_{j=1}^n f(j)| = +\infty.$$
One can ask (as was done in \cite{erdos}) whether the same claim holds if $f$ is only assumed to be multiplicative rather than completely multiplicative.  As noted in \cite{coons}, there is a simple counterexample, namely the multiplicative function $\chi_2: \N \to \{-1,+1\}$ defined by setting $\chi_2(n) := +1$ when $n$ is odd and $\chi_2(n) := -1$ when $n$ is even.  A bit more generally, any function of the form $f = \chi_2 h$ is a counterexample, where $h: \N \to \{-1,+1\}$ is a multiplicative function such that $h(2^j)=1$ for all $j$, and $h(p^j)=1$ for all but finitely many prime powers $p^j$, since such functions are periodic with mean zero thanks to the Chinese remainder theorem.  In the converse direction, the methods of this paper can be used to show

\begin{theorem}  Let $f: \N \to \{-1,+1\}$ be a multiplicative function such that $\sup_n |\sum_{j=1}^n f(j)| < +\infty$.  Then $f(2^j) = -1$ for all $j$, and
$$ \sum_p \frac{1-f(p)}{p} < \infty.$$
\end{theorem}

Informally, this theorem asserts that the only multiplicative functions $f: \N \to \{-1,+1\}$ with bounded sums are ones which ``pretend'' to be $\chi_2$.
In \cite{coons}, it was shown that if $f: \N \to \{-1,+1\}$ is a multiplicative function with $f(2^j)=1$ for some natural number $j$, and one had the asymptotic $\sum_{p\leq x} f(p) = (c+o(1)) \log x$ for some $0 < c \leq 1$, then $\sup_n |\sum_{j=1}^n f(j)| = +\infty$.  This is implied by the above theorem.  It seems likely that one can carry the analysis further and conclude that the periodic examples given above are the complete list of multiplicative $f: \N \to \{-1,+1\}$ with $\sup_n |\sum_{j=1}^n f(j)| < +\infty$, but we will not do so here.

We sketch the proof of the above theorem as follows.  Suppose that we have a multiplicative function $f: \N \to \{-1,+1\}$ such that $|\sum_{j=1}^n f(j)| \leq C$ for all $n$ and some finite $C$.  Henceforth we allow implied constants to depend on $C$. Applying Proposition \ref{star} (ignoring the stochasticity and the $\eps$ parameter, and generalising from completely multiplicative functions to multiplicative functions as in Remark \ref{gil}), we see that for each $X \geq 1$, one can find a Dirichlet character $\chi$ of period $q = O(1)$ and a real number $t = O(X)$ such that
$$ \sum_{p \leq X} \frac{1 - \operatorname{Re} f(p) \overline{\chi}(p) p^{-it}}{p} \ll 1.$$
Actually, since $f$ is real-valued, we may take $t=0$ and $\chi$ to be real-valued by the triangle inequality argument in \cite[Appendix C]{mrt}.  Thus
$$ \sum_{p \leq x} \frac{1 - f(p) \chi(p)}{p} \ll 1.$$
Currently, $\chi$ is allowed to depend on $x$, but the number of possible $\chi$ is bounded independently of $x$, and so by the pigeonhole principle (or compactness) we can thus find a Dirichlet character $\chi$ of period $q = O(1)$ such that
$$ \sum_{p} \frac{1 - f(p) \chi(p)}{p} \ll 1.$$
As before, we may assume without loss of generality that $\chi$ is primitive.  We then factor
$$ f = \tilde \chi h$$
where $\tilde \chi$ is the multiplicative function such that $\tilde \chi(n) := \chi(n)$ for $(n,q)=1$ and $\tilde \chi(p^j) := f(p^j)$ for $p|q$ and $j \geq 1$, and $h = f / \tilde \chi$ is also multiplicative taking values in $\{-1,+1\}$, with $h(p^j)=1$ whenever $p|q$ and $j \geq 1$, and
\begin{equation}\label{hpo}
 \sum_{p} \frac{1 - h(p)}{p} \ll 1.
\end{equation}
We allow implied constants to depend on $q, \chi, h$.
Suppose first that $h(2^j) = +1$ for at least one natural number $j$.  Then the Euler factors $\sum_{j=0}^\infty \frac{h(p^j)}{p^j}$ are non-zero for every prime $p$ (the only dangerous case being $p=2$), and one can then check from \eqref{hpo} and Mertens' theorem that the singular series
$$ {\mathfrak S} := \prod_p \sum_{j=0}^\infty \frac{h(p^j)}{p^{j(1+1/\log X)}}$$
obeys the bounds
$$ \log X \ll {\mathfrak S} \ll \log X$$
for all sufficiently large $X$ (recall we allow implied constants to depend on $h$).  One can then check that the argument in Section \ref{bcc-sec} (ignoring the stochasticity, the $\eps$ parameter, and the complex conjugations, and setting $t$ to zero) continues to work with very little modification (even though $\tilde \chi$ and $h$ are now only multiplicative rather than completely multiplicative) to give the desired contradiction for suitable choices of parameters $H, k, X$ as before (the $\delta$ parameter is irrelevant, since Lemma \ref{tb} is automatic in this setting).  Thus we may assume that $h(2^j)=-1$ for all $j$, which implies in particular that $q$ is odd since $h(p)=+1$ for all $p|q$.

The Euler product ${\mathfrak S}$ now vanishes due to the $p=2$ factor, which prevents us from applying the arguments in Section \ref{bcc-sec} immediately.  To get around this, we now factor
$$ f = \tilde \chi' h'$$
where $\tilde \chi' := \chi_2 \tilde \chi$ and $h' := \chi_2 h$.  One can check that the singular series
$$ {\mathfrak S}' := \prod_p \sum_{j=0}^\infty \frac{h'(p^j)}{p^{j(1+1/\log X)}}$$
obeys the bounds
$$ \log X \ll {\mathfrak S}' \ll \log X$$
for sufficiently large $X$.  If $q \neq 1$, one can then run the previous arguments with $\tilde \chi$, $h$ replaced by $\tilde \chi'$, $h'$ respectively (and $q^k$ replaced by $2q^k$), arriving at the analogue
$$
\frac{1}{2q^k} \sum_{a \in [1,2q^k] \hbox{ good}} \frac{1}{H} \sum_{H < H' \leq 2H} \left|\sum_{m=1}^{H'} \tilde{\chi}'(a+m)\right|^2 \ll 1
$$
of \eqref{contra}. But we can write $\tilde \chi'(a+m) = -\chi_2(a) \chi_2(m) \tilde \chi(a+m)$, and hence
\begin{equation}\label{contrad}
\frac{1}{2q^k} \sum_{a \in [1,2q^k] \hbox{ good}} \frac{1}{H} \sum_{H < H' \leq 2H} \left|\sum_{m=1}^{H'} \chi_2(m) \tilde{\chi}(a+m)\right|^2 \ll 1
\end{equation}
and by repeating the rest of the arguments in Section \ref{bcc-sec} (carrying along the $\chi_2(m)$ and $2$ factors which end up being harmless) we again obtain a contradiction.  Thus $q=1$, and the theorem follows.

%%% AUTHOR: optional acknowledgments here
\section*{Acknowledgments} %%  you may comment this out if no Ackno

The author is supported by NSF grant DMS-0649473 and by a Simons Investigator Award. The author thanks Uwe Stroinski for suggesting a possible connection between Elliott-type results and the Erd\H{o}s discrepancy problem, leading to the blog post at {\tt terrytao.wordpress.com/2015/09/11} in which it was shown that a (non-averaged) version of the Elliott conjecture implied Theorem \ref{hilbert}.  Shortly afterwards, the author obtained the averaged version of that conjecture in \cite{tao-elliott}, which turned out to be sufficient to complete the argument.  The author also thanks Timothy Gowers for helpful discussions and encouragement, as well as Crist\'obal Camarero, Christian Elsholtz, Andrew Granville, Gergely Harcos, Gil Kalai, Joseph Najnudel, Royce Peng, Uwe Stroinski, and anonymous blog commenters for corrections and comments on the above-mentioned blog post and on other previous versions of this manuscript.  Finally, we thank the anonymous referee for a thorough reading of the manuscript and for many comments and corrections.

%%% AUTHOR:
%%% Bibliography goes here. Note that the arXiv cannot process bibtex
%%% or biber bibliographies.  Example of acceptable bibliograpy format:

%% AUTHOR: You can generate such a bibliography from a .bib file by 
%% running pdflatex/bibtex/pdflatex/pdflatex and then pasting the .bbl file
%% between \begin{thebibliography} and \end{bibliography}

%%% AUTHOR: Include a short description of each author following the
%%% structure below. Use the same short tags used previously.  
%%% Use \imageat{} and \imagedot{} instead of "@" and "." in
%%% email addresses-this replaces the symbols with graphics to avoid 
%%% e-mail address harvesting from the .pdf file
\begin{dajauthors}
\begin{authorinfo}[tt]
  Terence Tao\\
  Department of Mathematics, UCLA\\
  405 Hilgard Ave\\
  tao\imageat{}math.ucla.edu\\
	\url{https://www.math.ucla.edu/~tao}
\end{authorinfo}
\end{dajauthors}

\end{document}